\documentclass[11pt,fleqn]{scrartcl}

\usepackage[twoside,
  paperwidth=210mm,
  paperheight=297mm,
  textheight=682pt,
  textwidth=520pt,
  centering,
  headheight=50pt,
  headsep=12pt,
  footskip=18pt,
  footnotesep=24pt plus 2pt minus 12pt,
  columnsep=18pt]{geometry}

\usepackage{hyperref}

\usepackage{amsthm}
\usepackage{mathtools}
\usepackage{amssymb}

\DeclareMathAlphabet{\mathfrak}{U}{euf}{m}{n}
\SetMathAlphabet{\mathfrak}{bold}{U}{euf}{b}{n}

\DeclareSymbolFont{AMSb}{U}{msb}{m}{n}
\DeclareSymbolFontAlphabet{\mathbb}{AMSb}

\newtheorem{theorem}{Theorem}[section]
\newtheorem{lemma}[theorem]{Lemma}
\newtheorem{proposition}[theorem]{Proposition}
\newtheorem{corollary}[theorem]{Corollary}
\theoremstyle{remark}
\newtheorem{remark}[theorem]{Remark}
\theoremstyle{definition}
\newtheorem{definition}[theorem]{Definition}

\usepackage{lmodern}
\usepackage[utf8]{inputenc}
\usepackage[T1]{fontenc}

\DeclareMathSymbol{\Gamma}{\mathalpha}{letters}{0}
\DeclareMathSymbol{\Delta}{\mathalpha}{letters}{1}
\DeclareMathSymbol{\Theta}{\mathalpha}{letters}{2}
\DeclareMathSymbol{\Lambda}{\mathalpha}{letters}{3}
\DeclareMathSymbol{\Xi}{\mathalpha}{letters}{4}
\DeclareMathSymbol{\Pi}{\mathalpha}{letters}{5}
\DeclareMathSymbol{\Sigma}{\mathalpha}{letters}{6}
\DeclareMathSymbol{\Upsilon}{\mathalpha}{letters}{7}
\DeclareMathSymbol{\Phi}{\mathalpha}{letters}{8}
\DeclareMathSymbol{\Psi}{\mathalpha}{letters}{9}
\DeclareMathSymbol{\Omega}{\mathalpha}{letters}{10}

\DeclareMathSymbol{\alpha}{\mathalpha}{letters}{11}
\DeclareMathSymbol{\beta}{\mathalpha}{letters}{12}
\DeclareMathSymbol{\gamma}{\mathalpha}{letters}{13}
\DeclareMathSymbol{\delta}{\mathalpha}{letters}{14}
\DeclareMathSymbol{\epsilon}{\mathalpha}{letters}{15}
\DeclareMathSymbol{\zeta}{\mathalpha}{letters}{16}
\DeclareMathSymbol{\eta}{\mathalpha}{letters}{17}
\DeclareMathSymbol{\theta}{\mathalpha}{letters}{18}
\DeclareMathSymbol{\iota}{\mathalpha}{letters}{19}
\DeclareMathSymbol{\kappa}{\mathalpha}{letters}{20}
\DeclareMathSymbol{\lambda}{\mathalpha}{letters}{21}
\DeclareMathSymbol{\mu}{\mathalpha}{letters}{22}
\DeclareMathSymbol{\nu}{\mathalpha}{letters}{23}
\DeclareMathSymbol{\xi}{\mathalpha}{letters}{24}
\DeclareMathSymbol{\pi}{\mathalpha}{letters}{25}
\DeclareMathSymbol{\rho}{\mathalpha}{letters}{26}
\DeclareMathSymbol{\sigma}{\mathalpha}{letters}{27}
\DeclareMathSymbol{\tau}{\mathalpha}{letters}{28}
\DeclareMathSymbol{\upsilon}{\mathalpha}{letters}{29}
\DeclareMathSymbol{\phi}{\mathalpha}{letters}{30}
\DeclareMathSymbol{\chi}{\mathalpha}{letters}{31}
\DeclareMathSymbol{\psi}{\mathalpha}{letters}{32}
\DeclareMathSymbol{\omega}{\mathalpha}{letters}{33}
\DeclareMathSymbol{\varepsilon}{\mathalpha}{letters}{34}
\DeclareMathSymbol{\vartheta}{\mathalpha}{letters}{35}
\DeclareMathSymbol{\varpi}{\mathalpha}{letters}{36}
\DeclareMathSymbol{\varphi}{\mathalpha}{letters}{39}
\DeclareMathSymbol{\varrho}{\mathalpha}{letters}{37}
\DeclareMathSymbol{\varsigma}{\mathalpha}{letters}{38}

\begin{document}

\title{On a generalized pseudorelativistic Schr\"{o}dinger equation with supercritical growth}
\author{Simone Secchi\thanks{Supported by the MIUR 2015 PRIN project ``Variational methods, with applications to problems in mathematical physics and geometry'' and by INdAM through a GNAMPA 2017 project.} \\ {\fontsize{9pt}{12pt}\selectfont Dipartimento di Matematica e Applicazioni, Universit\`a di Milano Bicocca} \\ {\fontsize{9pt}{12pt}\selectfont via R. Cozzi 55, Milano (Italy).} \\ {\fontsize{9pt}{12pt}\selectfont \texttt{Simone.Secchi@unimib.it}}}

\maketitle

\hfill {\itshape To Francesca, with love}

\bigskip
	
\begin{abstract}
	\noindent We prove that the generalized  pseudorelativistic equation 
\begin{displaymath}
\left( -c^2 \Delta + m^{2} c^{\frac{2}{1-s}} \right)^{s} u - m^{2s} c^{\frac{2s}{1-s}} u  +\mu u = |u|^{p-1} u
\end{displaymath}
can be solved for large values of the ``light speed'' $c$ even when
$p$ crosses the critical value for the fractional Sobolev embedding.


\end{abstract}

\section{Introduction}

The \emph{pseudorelativistic Schr\"{o}dinger equation}
\begin{equation} \label{eq:PSE}
\mathrm{i} \frac{\partial \psi}{\partial t} = \sqrt{\strut -c^2 \Delta +m^2 c^4} \psi - mc^2 \psi + f(|\psi|^2)\psi,
\end{equation}
in which $c$ denotes the speed of the light, $m>0$ represents the particle mass and $f \colon [0,\infty) \to \mathbb{R}$ is a nonlinear function, is one of the relativistic versions of the more familiar NLS
\begin{equation*}
\mathrm{i} \frac{\partial \psi}{\partial t} = - \frac{1}{2m} \Delta  \psi  + f(|\psi|^2)\psi.
\end{equation*}
Equation (\ref{eq:PSE}) describes, from the physical viewpoint, the dynamics of systems consisting of identical spin-0 bosons whose motions are relativistic, like \emph{boson stars}. We refer to \cite{DallAcqua08,Elgart07,Frohlich07,Lenzmann07,Lieb84,Lieb87} for the rigorous derivation of the equation and the study of its dynamical properties.

When $f(t)=-t^{\frac{p-1}{2}}$, standing waves $\psi(t,x)=\exp (\mathrm{i} \mu t) u(x)$ must satisfy the stationary equation
\begin{equation} \label{eq:PSEbis}
	\sqrt{\strut -c^2 \Delta + m^2 c^4}u -m c^2 u + \mu u = |u|^{p-1} u.
\end{equation}
Coti Zelati and Nolasco showed in \cite{Coti11} that for $c=1$ and $1<p<(N+1)/(N-1)$ there exists a radially symmetric positive solution to \eqref{eq:PSEbis}. This result was extended later in \cite{Choi16}. Another variant of the above equation is the pseudorelativistic Hartree equation
\begin{equation*}
\sqrt{\strut -c^2 \Delta + m^2 c^4} u - m c^2 u + \mu u = \left( I_\alpha \star |u|^{p} \right) |u|^{p-2}u,
\end{equation*}
where $I_\alpha \colon \mathbb{R}^N \setminus \{0\} \to \mathbb{R}$ is
a singular convolution kernel. If, formally, $I_\alpha$ degenerates to
a Dirac delta, the Hartree equation reduces to \eqref{eq:PSE}. The
case in which $I_\alpha (x) = |x|^{N-\alpha}$ is particularly
important. We refer to \cite{Secchi16} for a survey of recent results.

\bigskip

As already noticed, the application of variational techniques to
(\ref{eq:PSEbis}) requires a bound from above on the exponent $p$,
since the natural Sobolev space in which (\ref{eq:PSEbis}) can be set
is $H^{1/2}(\mathbb{R}^N)$ and this space is embedded into
$L^{q}(\mathbb{R}^N)$ only if $q \leq 2N/(N-1)$. Local compactness of
the embedding exludes the limiting exponent, and therefore it is
customary to assume that $1<p<(N+1)/(N-1)$.

On the other hand, if we observe that the pseudorelativistic operator
\(\sqrt{\strut -c^2 \Delta + m^2 c^4} - mc^2\) \emph{converges} to
$-\frac{1}{2m} \Delta$ as $c \to +\infty$, we may expect that
solutions could exist for $c \gg 1$ as soon as $p < (N+2)/(N-2)$,
namely below the critical Sobolev exponent for the
operator~$-\frac{1}{2m}\Delta+1$. This fact has been proved recently
in \cite{Choi17}.

\bigskip

In this paper we consider the generalized model
	\begin{equation}
	\label{eq:1}
	\left( -c^2 \Delta + m^{2} c^{\frac{2}{1-s}} \right)^{s} u - m^{2s} c^{\frac{2s}{1-s}} u  +\mu u = |u|^{p-1} u,
	\end{equation}
	where $1/2 < s <1$ and $p>1$, which reduces to \eqref{eq:PSEbis} for $s=1/2$. For $c=1$, this equation has been studied in \cite{Ambrosio16,Ikoma17,Secchi16-1,Secchi17,Secchi17-1}. Following the ideas of \cite{Choi17}, we  prove that (\ref{eq:1}) is actually solvable in the whole range $1<p<(N+2)/(N-2)$ in the \emph{r\'{e}gime} $c \gg 1$.
\begin{theorem} \label{th:main}
Let $N \geq 3$ and $1/2<s<1$. For every $p \in (1,\frac{N+2}{N-2})$, equation \eqref{eq:1} admits at least a nontrivial solution in $H_{\mathrm{rad}}^1 \cap L^\infty$ provided that $m^{2s} c^{\frac{2s}{1-s}}/\mu$ is sufficiently large.
\end{theorem}
\begin{remark}
The restriction $1/2\leq s<1$ is somehow natural, if we want bounded solutions. In the fractional framework, an exponent $s<1/2$ does not ensure high regularity properties of solutions.
\end{remark}
\begin{remark}
Although we have stated Theorem \ref{th:main} for any dimension~$N\geq 3$, it is rather easy to check that the same result holds also for~$N=1$ or $2$. In this case, however, the Sobolev critical exponent no longer exists.
\end{remark}
Our approach is based on the reduction of equation~\eqref{eq:1} to a fixed-point problem. By means of some \emph{pseudo-differential calculus} we can overcome the limitation of the variational setting.

\bigskip

We leave as an open problem the study of \eqref{eq:1} in the \emph{r\'{e}gime} $m^{2s}c^{\frac{2s}{1-s}}/\mu \ll 1$. Some results appear in \cite{Ambrosio16} when $c=1=\mu$ and, consequently, $m \to 0$. Anyway, since the limit equation is in this case
\begin{equation*}
(-\Delta)^s u + u =|u|^{p-1}u \quad \hbox{in $\mathbb{R}^N$},
\end{equation*}
we do not expect any improvement in the range of the exponent~$p$.

\section{Preliminaries and estimates}

It will be useful to collect some notation that we are going to employ throughout the paper.
\begin{itemize}
	\item An integral over the whole space $\mathbb{R}^N$ is denoted simply by $\int$ instead of $\int_{\mathbb{R}^N}$.
	\item The Fourier transform of a (suitably regular) function $\varphi$ is denoted by
	\begin{equation*}
	\widehat{\varphi} \colon \xi \mapsto (2 \pi)^{-N/2} \int e^{-\mathrm{i}\xi \cdot x} \varphi(x)\, dx.
	\end{equation*}
	\item Function spaces like $L^p(\mathbb{R}^N)$ or $W^{1,p}(\mathbb{R}^N)$ are denoted by $L^p$ and $W^{1,p}$. 
	\item If $X$ is some function space, we denote by $X_{\mathrm{rad}}$ the subspace of $X$ consisting of radially symmetric functions.
	\item The identity operator is denoted by $I$. Sometimes, however, we denote the multiplication operator against a given function $u$ by $u$ instead of $uI$.
	\item Derivatives are always denoted by the letter $D$ or by the symbol $\partial$. In particular, $\partial_{x_j}$ denotes the partial derivative with respect to the variable $x_j$. If $\alpha=(\alpha_1,\ldots,\alpha_N)$ is a multi-index, we denote $D^\alpha = \partial_{x_1}^{\alpha_1} \cdots \partial_{x_j}^{\alpha_j}$.
	\item Since we are interested in asymptotic estimates as $c \to +\infty$, we use the symbol $\lesssim$ to denote an inequality with a multiplicative constant independent of $c$. Therefore $a \lesssim b$ means that $a \leq C b$ with some constant $C>0$ that does not depend on $c$.
	\item For two given Banach spaces, we denote by $L(X,Y)$ the space of continuous linear operators from $X$ to $Y$. The norm in $L(X,Y)$ is the usual one: $\| \Lambda \|_{L(X,Y)} = \sup_{\|x\|_X=1} \|\Lambda x \|_Y$.
	\item If $\Lambda$ is an invertible linear operator, we sometimes use the ``algebraic'' piece of notation $1/\Lambda$ to denote the inverse $\Lambda^{-1}$.
\end{itemize}
  We begin with an ``almost necessary'' condition so that solutions to \eqref{eq:1} may exist. This is also a motivation for our attempt to construct a solution in a suitable \emph{supercritical} setting.
\begin{theorem}
If $p \geq \frac{N+2s}{N-2s}$ but $m^{2s} c^{\frac{2s}{1-s}} \leq \mu$, there is no non-trivial solution $u \in H^{s} (\mathbb{R}^N) \cap L^\infty(\mathbb{R}^N)$.
\end{theorem}
\begin{proof}
Let us consider, for simplicity, the equation
\begin{equation} \label{eq:poho}
(a^2 - b^2 \Delta)^s u = f(u) \quad \text{in $\mathbb{R}^N$},
\end{equation}
where $f$ is a suitable nonlinearity. For a smooth function $\varphi$, we recall that
\begin{equation*}
\mathcal{F} (\partial_{x_k} \varphi) = \mathrm{i} \xi_k \hat{\varphi}, \qquad \mathcal{F}(x_k \varphi) = \mathrm{i} \partial_{\xi_k} \hat{\varphi}.
\end{equation*}
This implies that
\begin{equation*}
\mathcal{F}(x \cdot \nabla \varphi) \colon \xi \mapsto - \left( \xi \cdot \nabla \hat{\varphi} + N \hat{\varphi} \right)
\end{equation*}
and
\begin{equation*}
\mathcal{F} \left( (a^2-b^2 \Delta)^s (x \cdot \nabla \varphi) \right) \colon \xi \mapsto - (a^2+b^2 |\xi|^2)^s \left( \xi \cdot \nabla \hat{\varphi} + N \hat{\varphi} \right).
\end{equation*}
Hence
\begin{multline*}
\mathcal{F} \left( x \cdot \nabla (a^2-b^2 \Delta)^s \varphi \right) \colon \xi \mapsto - \left( \xi \cdot \nabla \mathcal{F} \left( (a^2-b^2 \Delta)^s \varphi \right) + N \mathcal{F} \left( (a^2-b^2  \Delta)^s \varphi \right) \right) \\
= - \left( \xi \cdot \nabla \left( (a^2+b^2 |\xi|^2 )^s \hat{\varphi} \right) + N (a^2+b^2 |\xi|^2)^s \hat{\varphi} \right) \\
= -(a^2 + b^2 |\xi|^2 )^s \left( \xi \cdot \nabla \hat{\varphi} + N \hat{\varphi} \right) - 2 sb^2 (a^2+b^2 |\xi|^2)^{s-1} |\xi|^2 \hat{\varphi} \\
= -(a^2 + b^2 |\xi|^2 )^s \left( \xi \cdot \nabla \hat{\varphi} + N \hat{\varphi} \right)  -2s (a^2+b^2|\xi|^2)^s \hat{\varphi} + 2sa^2 (a^2+b^2 |\xi|^2)^{s-1} \hat{\varphi}.
\end{multline*}
We have proved the following \emph{pointwise identity}:
\begin{equation} \label{eq:pohopunt}
(a^2-b^2 \Delta)^s (x \cdot \nabla \varphi) - x \cdot \nabla \left( (a^2-b^2 \Delta)^s \varphi \right) = 2s (a^2-b^2 \Delta)^s \varphi - 2a^2 s (a^2-b^2 \Delta)^{s-1} \varphi.
\end{equation}
In the rest of the proof, we will be somehow sketchy. For a rigorous argument, we should replace $u$ with $u_\varepsilon = \rho_\varepsilon \star u$, where $\rho_\varepsilon$ is a mollifier, and then take the limit as $\varepsilon \to 0$. We omit the technical details. Using \eqref{eq:pohopunt} we get
\begin{multline*}
\int (a^2 -b^2 \Delta)^s u (x \cdot \nabla u) = \int u (a^2-b^2 \Delta)^s (x \cdot \nabla u) = \int u \left( x \cdot \nabla f(u) + 2s f(u) - 2a^2 s (a^2-b^2 \Delta)^{s-1}u \right) \\
= -N \int u f(u) + N \int F(u) + \int 2s uf(u) - \int 2a^2 s u(a^2-b^2 \Delta)^{s-1}u.
\end{multline*} 
Since
\begin{equation*}
\int f(u) \left( x \cdot \nabla u \right) = -N \int F(u),
\end{equation*}
we find the identity
\begin{equation} \label{eq:pohozaev}
(N-2s) \int u f(u) - 2N \int F(u) + 2a^2 s \int u (a^2 -b^2 \Delta)^{s-1}u.
\end{equation}
We observe that
\begin{equation*}
\int u (a^2 -b^2 \Delta)^{s-1}u = \int (a^2 +b^2 |\xi|^2)^{s-1} |\hat{u}|^2 \geq 0.
\end{equation*}
We now choose
\begin{equation*}
f(s) = |s|^{p-1}s - \kappa s
\end{equation*}
with $\kappa \in \mathbb{R}$. Then $F(s) = \frac{1}{p+1} |s|^{p+1} - \frac{\kappa}{2} s^2$, and \eqref{eq:pohozaev} yields
\begin{equation*}
\left( \frac{1}{p+1} - \frac{N-2s}{2N} \right) \int |u|^{p+1} = \frac{\kappa s}{N} \int |u|^2 + \frac{a^2 s}{N} \int u (a^2-b^2\Delta)^{s-1}u.
\end{equation*}
If $p \geq \frac{N+2s}{N-2s}$ and $\kappa \geq 0$, then $u=0$.

We apply this conclusion with
\begin{equation*}
a^2 = m^2 c^{\frac{2}{1-s}}, \quad b^2 = c^2, \quad \kappa = \mu - m^{2s} c^{\frac{2s}{1-s}}
\end{equation*}
to prove our result.
\end{proof}
On the other hand, we are going to construct a (non-trivial) solution to \eqref{eq:1} when the quotient $m^{2s} c^{\frac{2s}{1-s}}/\mu$ is sufficiently large.
As a first step, we show that some parameters in~\eqref{eq:1} are not essential.  It is anyway clear that the pseudorelativistic Schr\"{o}dinger equation is \emph{not} scale-invariant, but this obstacle is irrelevant in the limit $c \to +\infty$.
\begin{proposition}
The equation
\begin{equation*}
\left( \left( -\tilde{c}^2 \Delta + m^2 \tilde{c}^{\frac{2}{1-s}} \right)^s - m^{2s} \tilde{c}^{\frac{2s}{1-s}} \right) v + \mu v = |v|^{p-1}v
\end{equation*}
is equivalent to  
\begin{equation} \label{eq:36}
\left( -c^2 \Delta +s^{1-s} c^{\frac{2}{1-s}} \right)^s  u - s^{\frac{s}{1-s}} c^{\frac{2s}{1-s}} u + u = |u|^{p-1}u
\end{equation}
for a suitable choice of~$c$ and~$\tilde{c}$.
\end{proposition}
\begin{proof}
We suppose that $v$ satisfies (at least formally) the equation
\begin{equation*}
\left( \left( -\tilde{c}^2 \Delta + m^2 \tilde{c}^{\frac{2}{1-s}} \right)^s - m^{2s} \tilde{c}^{\frac{2s}{1-s}} \right) v + \mu v = f
\end{equation*}
with $f=|v|^{p-1}v$.
Let us set
\[
u(x) = \rho v \left( \sigma x \right)
\]
for some constants $\rho>0$, $\sigma>0$, so that
\[
\widehat{u}(\xi) = \frac{\rho}{\sigma^N} \widehat{v} \left( \frac{\xi}{\sigma} \right).
\] 
Then, by using Fourier variables,
\begin{equation*}
\left( \tilde{c}^2 |\xi|^2 + m^2 \tilde{c}^{\frac{2}{1-s}} \right)^s \tilde{v} - m^{2s} \tilde{c}^{\frac{2s}{1-s}} \tilde{v} + \mu \tilde{v} = \tilde{f}.
\end{equation*}
Equivalently,
\begin{equation*}
\left( \tilde{c}^2 \left| \frac{\xi}{\sigma} \right|^2 + m^2 \tilde{c}^{\frac{2}{1-s}} \right)^s \tilde{v}\left( \frac{\xi}{\sigma} \right) - m^{2s} \tilde{c}^{\frac{2s}{1-s}} \tilde{v}\left( \frac{\xi}{\sigma} \right) + \mu \tilde{v}\left( \frac{\xi}{\sigma} \right) = \tilde{f}\left( \frac{\xi}{\sigma} \right).
\end{equation*}
If we multiply through by $\rho/\sigma$ and move back to Euclidean variables, recalling that $f=|v|^{p-1}v$,
\begin{equation*}
\left( - \frac{\tilde{c}^2}{\sigma^2} \Delta + m^2 \tilde{c}^{\frac{2}{1-s}} \right)^s u - m^{2s} \tilde{c}^{\frac{2s}{1-s}} u + \mu u = \rho^{1-p} |u|^{p-1}u.
\end{equation*}
Now we choose
\begin{equation*}
	\rho =\mu^{\frac{1}{1-p}}, \quad
	\tilde{c} = \frac{\sqrt{s} \mu^{\frac{1-s}{2s}}}{m^{1-s}} c, \quad
	\sigma^2 = \frac{s}{\mu m^{2(1-s)}}.
\end{equation*}
We remark in particular that $\tilde{c} \to +\infty$ if and only if $c \to +\infty$.
After some elementary but lengthy computation we can show that \eqref{eq:1} reduces to
\begin{equation*} 
\left( -c^2 \Delta + \frac{1}{s^{\frac{1}{s-1}}} c^{\frac{2}{1-s}} \right)^s  u - \frac{1}{s^{\frac{s}{s-1}}} c^{\frac{2s}{1-s}} u + u = |u|^{p-1}u,
\end{equation*}
that is \eqref{eq:36}.
\end{proof}
As a consequence, it is not restrictive to assume that \(m=s^{1/(2-2s)}\) and $\mu=1$,
and we will consider equation~(\ref{eq:36}) in the rest of the paper.

We now introduce the (pseudodifferential) operators
\begin{align*}
	P_c(D) &= \left(  \left( c^2 |D|^2 + \frac{1}{s^{\frac{1}{s-1}}} c^{\frac{2}{2-s}} \right)^s-  \frac{1}{s^{\frac{s}{s-1}}} c^{\frac{2s}{1-s}}
	\right) + 1 \\
	P_\infty (D)&= |D|^2 + 1,
\end{align*}
which are  associated to the symbols
\begin{align*}
	P_c(\xi) &= \left(  \left( c^2 |\xi|^2 + \frac{1}{s^{\frac{1}{s-1}}} c^{\frac{2}{2-s}} \right)^s-  \frac{1}{s^{\frac{s}{s-1}}} c^{\frac{2s}{1-s}}
	\right) + 1  \\
	P_\infty (\xi)&= |\xi|^2 + 1.
\end{align*}
\begin{remark}
In general, we recall that given a symbol $\mathfrak{m} \colon \mathbb{R}^N \to \mathbb{R}$, the associated Fourier multiplier operator $\mathfrak{m}(D)$ is defined (on smooth functions) by
\begin{equation*}
\widehat{\mathfrak{m}(D)} f(\xi) = \mathfrak{m}(\xi) \hat{f}(\xi).
\end{equation*}
\end{remark}
With this notation our problem is equivalent to
\begin{equation*}
P_c(D)u = |u|^{p-1}u
\end{equation*}
as $c \gg 1$. The main idea is to begin with a solution $u_\infty$ of
\begin{equation*}
P_\infty(D)(u_\infty)=|u_\infty|^{p-1}u_\infty
\end{equation*}
and set $w=u-u_\infty$. Hence
\begin{gather*}
P_c(D)w = P_c(D)u-P_c(D)u_\infty = P_c(D)u + P_\infty(D)u_\infty - P_\infty(D)u_\infty - P_c(D) u_\infty \\
= \left( P_\infty(D)-P_c(D)\right) u_\infty + P_c(D)u-P_\infty(D)u_\infty \\
= \left( P_\infty(D)-P_c(D)\right) u_\infty + |w+u_\infty|^{p-1}(w+u_\infty) - u_\infty^p.
\end{gather*}
Equivalently,
\begin{equation*}
\mathcal{L}_{c,\infty} w = \left(P_\infty(D)-P_c(D) \right) u_\infty + Q(w),
\end{equation*}
where
\begin{equation*}
\mathcal{L}_{c,\infty} = P_c(D)-pu_\infty^{p-1}
\end{equation*}
and
\begin{align*}
Q(w) = |w+u_\infty|^{p-1}(w+u_\infty) - u_\infty^p - pu_\infty^{p-1}w.
\end{align*}
If we can invert $\mathcal{L}_{c,\infty}$ in a suitable space, then we may write the fixed-point equation
\begin{equation*}
w = \left(\mathcal{L}_{c,\infty} \right)^{-1} \left( P_\infty(D)-P_c(D) \right)u_\infty + \left(\mathcal{L}_{c,\infty} \right)^{-1} Q(w). 
\end{equation*}
We will prove that the nonlinear operator
\begin{equation} \label{eq:30}
\Phi_c(w) = \left(\mathcal{L}_{c,\infty} \right)^{-1} \left( P_\infty(D)-P_c(D) \right)u_\infty + \left(\mathcal{L}_{c,\infty} \right)^{-1} Q(w)
\end{equation}
is contractive in a small ball of a suitable Sobolev space. To do this we need, as expected, some careful estimates.
\begin{lemma} \label{lem:1.1}
There results
\begin{equation} \label{eq:10}
\frac{|\xi|^2+1}{2} \leq P_c(\xi) \leq  |\xi|^2 \quad\hbox{if $|\xi| \leq c^{\frac{s}{1-s}} \sqrt{\strut
\frac{2^{\frac{1}{1-s}}-1}
{s^{\frac{1}{s-1}}}
}$},
\end{equation}
and
\begin{equation} \label{eq:11}
\left( \left( 1+\frac{1}{2^{\frac{1}{1-s}}-1} \right)^s - \frac{1}{\left(2^{\frac{1}{1-s}}-1 \right)^s} \right) c^{2s} |\xi|^{2s}+1 \leq P_c(\xi) \leq c^{2s} |\xi|^{2s} +1
\end{equation}
if $|\xi| \geq c^{\frac{s}{1-s}} \sqrt{\strut
	\frac{2^{\frac{1}{1-s}}-1}
	{s^{\frac{1}{s-1}}}
}$.
\end{lemma}
\begin{proof}
We observe that
\begin{align*}
P_c(\xi) &= \left(  \left( c^2 |\xi|^2 + \frac{1}{s^{\frac{1}{s-1}}} c^{\frac{2}{1-s}} \right)^s-  \frac{1}{s^{\frac{s}{s-1}}} c^{\frac{2s}{1-s}}
	\right) + 1 \\
	&= \frac{c^{\frac{2s}{1-s}}}{s^{\frac{s}{s-1}}} \left( \left( 1+ s^{\frac{1}{s-1}} c^{-\frac{2s}{1-s}} |\xi|^2 \right)^s -1\right) +1 \\
&=  \frac{c^{\frac{2s}{1-s}}}{s^{\frac{s}{s-1}}} f\left( s^{\frac{1}{s-1}} \frac{|\xi|^2}{c^{\frac{2s}{1-s}}} \right)+1,
\end{align*}
where
\[
f(t) = \left(1+t \right)^s-1 \quad\hbox{for every $t \geq 0$}.
\]
Clearly $f(0)=0$, and 
\[
Df(t) = \frac{s}{(1+t)^{1-s}}.
\]
Plainly $Df(t) \leq s$ for every $t \geq 0$, while
\begin{equation} \label{eq:dalbasso}
\frac{s}{(1+t)^{1-s}} \geq \frac{s}{2} \quad\hbox{if $t \leq 2^{\frac{1}{1-s}}-1$}.
\end{equation}
Assume that
\[
|\xi| \leq c^{\frac{s}{1-s}} \sqrt{\strut
	\frac{2^{\frac{1}{1-s}}-1}
	{s^{\frac{1}{s-1}}}
}
\]
so that, by Taylor's theorem 
\begin{align}
P_c(\xi) -1 &= \frac{c^{\frac{2s}{1-s}}}{s^{\frac{s}{s-1}}} f\left( s^{\frac{1}{s-1}} \frac{|\xi|^2}{c^{\frac{2s}{1-s}}} \right) \leq \frac{c^{\frac{2s}{1-s}}}{s^{\frac{s}{s-1}}}  \left( f(0) + Df(0) s^{\frac{1}{s-1}} \frac{|\xi|^2}{c^{\frac{2s}{1-s}}}  \right) \nonumber \\
&\leq  |\xi|^2. \label{eq:26}
\end{align}
Similarly, \eqref{eq:dalbasso} implies 
\begin{equation*}
P_c(\xi) -1 \geq \frac{1}{2}|\xi|^2.
\end{equation*}
Hence \eqref{eq:10} is proved. Assume on the contrary that
\[
|\xi| \geq c^{\frac{s}{1-s}} \sqrt{\strut
	\frac{2^{\frac{1}{1-s}}-1}
	{s^{\frac{1}{s-1}}}
},
\]
and write
\begin{align*}
P_c(\xi)-1 &= \left( c^2 |\xi|^2 + \frac{c^{\frac{2}{1-s}}}{s^{\frac{1}{s-1}}} \right)^s - \frac{c^{\frac{2s}{1-s}}}{s^{\frac{s}{s-1}}} \\
&= c^{2s}|\xi|^{2s} \left( 1+ \frac{c^{\frac{2}{1-s}}}{s^{\frac{1}{s-1}} c^2 |\xi|^2} \right)^s - \frac{c^{\frac{2s}{1-s}}}{s^{\frac{s}{s-1}}} \\ 
&=c^{2s} |\xi|^{2s} \left(
\left( 1+ \frac{c^{\frac{2}{1-s}}}{s^{\frac{1}{s-1}} c^2 |\xi|^2} \right)^s - \frac{c^{\frac{2s^2}{1-s}}}{s^{\frac{s}{s-1}}|\xi|^{2s}}
\right) \\
&= c^{2s} |\xi|^{2s} \left(
\left( 1+ \frac{c^{\frac{2s}{1-s}}}{s^{\frac{1}{s-1}} |\xi|^2} \right)^s - \frac{c^{\frac{2s^2}{1-s}}}{s^{\frac{s}{s-1}}|\xi|^{2s}}
\right).
\end{align*}
The concavity of the function $t \mapsto t^s$ implies easily that
\begin{equation*}
\left( 1+ \frac{c^{\frac{2s}{1-s}}}{s^{\frac{1}{s-1}} |\xi|^2} \right)^s - \frac{c^{\frac{2s^2}{1-s}}}{s^{\frac{s}{s-1}}|\xi|^{2s}} \leq 1.
\end{equation*}
On the other hand, by monotonicity, $|\xi| \geq c^{\frac{s}{1-s}} \sqrt{\strut	\frac{2^{\frac{1}{1-s}}-1}{s^{\frac{1}{s-1}}}}$
implies
\begin{align*}
\left( 1+ \frac{c^{\frac{2s}{1-s}}}{s^{\frac{1}{s-1}} |\xi|^2} \right)^s - \frac{c^{\frac{2s^2}{1-s}}}{s^{\frac{s}{s-1}}|\xi|^{2s}} 
&\geq \left( 1+ 
\frac{c^{\frac{2s}{1-s}}}
{s^{\frac{1}{s-1}} c^{\frac{2s}{1-s}} \frac{2^{\frac{1}{1-s}}-1}{s^\frac{1}{s-1}}}
 \right)^s 
- \frac{c^{\frac{2s^2}{1-s}}}
{s^{\frac{s}{s-1}} c^{\frac{2s^2}{1-s}} \frac{\left(2^{\frac{1}{1-s}}-1\right)^s}{s^{\frac{s}{s-1}}}} \\
&= \left( 1+\frac{1}{2^{\frac{1}{1-s}}-1} \right)^s - \frac{1}{\left(2^{\frac{1}{1-s}}-1 \right)^s}.
\end{align*}
This proves \eqref{eq:11}.
\end{proof}
\begin{lemma} \label{lem:1.2}
	There results
	\begin{equation*}
	\left| P_c(\xi)-P_\infty(\xi) \right| \leq \frac{s(s-1)}{2s^{\frac{2-s}{1-s}}} \frac{|\xi|^4}{c^{\frac{2s}{1-s}}}
	\end{equation*}
\end{lemma}
\begin{proof}
	By direct computation, and using again the same notation as in Lemma \ref{lem:1.1},
	\begin{multline*}
	P_c(\xi)-P_\infty(\xi) = \left( c^2 |\xi|^2 + \frac{c^{\frac{2}{1-s}}}{s^{\frac{1}{s-1}}} \right)^s - \frac{c^{\frac{2s}{1-s}}}{s^{\frac{s}{s-1}}}-|\xi|^2 = \frac{c^{\frac{2s}{1-s}}}{s^{\frac{s}{s-1}}} f \left( s^{\frac{1}{s-1}} \frac{|\xi|^2}{c^{\frac{2s}{1-s}}} \right) - |\xi|^2 \\
	= \frac{c^{\frac{2s}{1-s}}}{s^{\frac{s}{s-1}}} \left( f(0) + Df(0) s^{\frac{1}{s-1}} \frac{|\xi|^2}{c^{\frac{2s}{1-s}}} + \frac{1}{2} D^2 f(\zeta) s^{\frac{2}{s-1}} \frac{|\xi|^4}{c^{\frac{4s}{1-s}}} \right) - |\xi|^2
	\end{multline*}
	for some $0< \zeta< s^{\frac{1}{s-1}} \frac{|\xi|^2}{c^{\frac{2s}{1-s}}}$. Since $D^2f(\zeta) = s(s-1)(\zeta+1)^{s-2}$, we conclude that
	\begin{align*}
	P_c(\xi)-P_\infty(\xi) &= \frac{c^{\frac{2s}{1-s}}}{s^{\frac{s}{s-1}}} \left( \frac{s^{\frac{s}{s-1}}}{c^{\frac{2s}{1-s}}} |\xi|^2 + \frac{1}{2} \frac{s(s-1)}{(\zeta+1)^{2-s}} \frac{s^{\frac{2}{s-1}}}{c^{\frac{4s}{1-s}}} |\xi|^4\right) - |\xi|^2 \\
	&= \frac{s(s-1)}{2 s^{\frac{2-s}{1-s}}(1+\zeta)^{2-s}} \frac{|\xi|^4}{c^{\frac{2s}{1-s}}}
	\leq \frac{s(s-1)}{2 s^{\frac{2-s}{1-s}}} \frac{|\xi|^4}{c^{\frac{2s}{1-s}}}.
	\end{align*}
\end{proof}
We can now prove the core result for our estimates.
\begin{proposition} \label{prop:1.4}
(a) For any multi-index $\alpha = (\alpha_1,\ldots,\alpha_N) \in (\mathbb{N} \cup \{0\})^N$ there exists a constant $C_\alpha>0$ such that for all $c \geq 2$,
\begin{equation} \label{eq:22}
\left| D^\alpha \left( \frac{1}{P_\infty(\xi)}-\frac{1}{P_c(\xi)} \right) \right| \leq \frac{C_\alpha}{|\xi|^{|\alpha|}} \min \left\{ \frac{1}{c^{\frac{2s}{1-s}}},\frac{1}{c^{\frac{2s^2}{1-s}}(|\xi|^2+1)^{1-s}} \right\}.
\end{equation}
(b) For any multi-index $\alpha = (\alpha_1,\ldots,\alpha_N) \in (\mathbb{N} \cup \{0\})^N$ there exists a constant $C_\alpha>0$ such that for all $c \geq 2$,
\begin{equation} \label{eq:25}
\left| D^\alpha \left( \frac{P_c(\xi)}{P_\infty(\xi)} \right) \right| \leq \frac{C_\alpha}{|\xi|^{|\alpha|}}
\end{equation}
\end{proposition}
\begin{proof}
Let us denote
\begin{equation*}
\mathfrak{a}(\xi) = \frac{1}{P_\infty(\xi)}-\frac{1}{P_c(\xi)} = 
\frac{P_c(\xi)-P_\infty(\xi)}{P_c(\xi) P_\infty(\xi)}.
\end{equation*}
According to the Leibniz rule for differentiation,
\begin{multline} \label{eq:16}
\partial_{\xi_j} \mathfrak{a}(\xi) = 
 \partial_{\xi_j} \left(  P_c(\xi)-P_\infty(\xi) \right) \frac{1}{P_c(\xi)} \frac{1}{P_\infty(\xi)} + \left(  P_c(\xi)-P_\infty(\xi) \right) \partial_{\xi_j} \left( \frac{1}{P_c(\xi)} \right) \frac{1}{P_\infty(\xi)} \\
 {} + \left(  P_c(\xi)-P_\infty(\xi) \right) \frac{1}{P_c(\xi)}  \partial_{\xi_j} \left( \frac{1}{P_\infty(\xi)} \right).
\end{multline}
Since
\begin{equation} \label{eq:24}
\begin{split}
&\partial_{\xi_j} \left(  P_c(\xi)-P_\infty(\xi) \right) = 2sc^2 \xi_j \left( c^2 |\xi|^2 + \frac{c^{\frac{2}{1-s}}}{s^{\frac{1}{s-1}}} \right)^{s-1}-2\xi_j \\
&\partial_{\xi_j} \left( \frac{1}{P_c(\xi)} \right) = -\frac{2sc^2 \xi_j \left(c^2|\xi|^2 + \frac{c^{\frac{2}{1-s}}}{s^{\frac{1}{s-1}}} \right)^{s-1}}{P_c(\xi)} \frac{1}{P_c(\xi)} \\
&\partial_{\xi_j} \left( \frac{1}{P_\infty(\xi)} \right) = -\frac{2\xi_j}{P_\infty(\xi)} \frac{1}{P_\infty(\xi)} \\
&\partial_{\xi_j} \left( c^2 |\xi|^2 + \frac{c^{\frac{2}{1-s}}}{s^{\frac{1}{s-1}}} \right)^{s-1} = - \frac{2sc^2 \xi_j}{\left(c^2 |\xi|^2 + \frac{c^{\frac{2}{1-s}}}{s^{\frac{1}{s-1}}}\right)\left(c^2 |\xi|^2 + \frac{c^{\frac{2}{1-s}}}{s^{\frac{1}{s-1}}}\right)^s},
\end{split}
\end{equation}
we conclude from \eqref{eq:16} that 
\begin{multline} \label{eq:17}
\partial_{\xi_j} \mathfrak{a}(\xi) = \left( 2sc^2 \xi_j \left( c^2 |\xi|^2 + \frac{c^{\frac{2}{1-s}}}{s^{\frac{1}{s-1}}} \right)^{s-1}-2\xi_j \right)\frac{1}{P_c(\xi)} \frac{1}{P_\infty(\xi)} \\ {} - \frac{2sc^2 \xi_j \left(c^2|\xi|^2 + \frac{c^{\frac{2}{1-s}}}{s^{\frac{1}{s-1}}} \right)^{s-1}}{P_c(\xi)} \mathfrak{a}(\xi) 
- \frac{2\xi_j}{P_\infty(\xi)} \mathfrak{a}(\xi),
\end{multline}
namely each partial derivative of $\mathfrak{a}$ is the sum of products of the following factors: $\mathfrak{a}$, $1/P_c$, $1/P_\infty$, $\big(c^2|\xi|^2 + \frac{c^{\frac{2}{1-s}}}{s^{\frac{1}{s-1}}} \big)^{s-1}$ and a polynomial of $c$, $\xi_1$, \ldots, $\xi_N$. If we iterate this procedure, we conclude that $D^\alpha \mathfrak{a}(\xi)$ can be expressed as a sum of products of:
\begin{center}
$\mathfrak{a}(\xi)$,  $P_c(\xi)^{-\ell_1}$, $P_\infty(\xi)^{-\ell_2}$, $\big(c^2|\xi|^2 + \frac{c^{\frac{2}{1-s}}}{s^{\frac{1}{s-1}}} \big)^{s-\ell_3}$, and a polynomial of $c$, $\xi_1$, \ldots, $\xi_N$,
\end{center}
where $\ell_1$, $\ell_2$ and $\ell_3$ are positive integers. 

Furthermore, from Lemma \ref{lem:1.1} we deduce that
\begin{equation} \label{eq:23}
	\begin{split}
&\left|
\frac{2sc^2 \xi_j \left(c^2|\xi|^2 + \frac{c^{\frac{2}{1-s}}}{s^{\frac{1}{s-1}}} \right)^{s-1}}{P_c(\xi)} \right| \leq
\begin{cases}
\frac{4}{|\xi|} &\text{if $|\xi| \leq c^{\frac{s}{1-s}} \sqrt{\strut
\frac{2^{\frac{1}{1-s}}-1}
{s^{\frac{1}{s-1}}}
}$} \\
&\frac{2s}{\left( \left( 1+\frac{1}{2^{\frac{1}{1-s}}-1} \right)^s - \frac{1}{\left(2^{\frac{1}{1-s}}-1 \right)^s} \right) |\xi|} \text{if $|\xi| \geq c^{\frac{s}{1-s}} \sqrt{\strut
\frac{2^{\frac{1}{1-s}}-1}
{s^{\frac{1}{s-1}}}
}$},
\end{cases}
\\
&\left| \frac{2 \xi_j}{P_\infty(\xi)} \right| \leq \frac{1}{|\xi|},
\\
&\left|
\frac{2sc^2 \xi_j}{c^2 |\xi|^2 + \frac{c^{\frac{2}{1-s}}}{s^{\frac{1}{s-1}}}}
\right| \leq \frac{2s |\xi|}{|\xi|^2+1} \leq \frac{s}{|\xi|}.
\end{split}
\end{equation}
We now recall Lemma \ref{lem:1.2} and estimate for $|\xi| \leq c^{\frac{s}{1-s}} \sqrt{\strut
\frac{2^{\frac{1}{1-s}}-1}
{s^{\frac{1}{s-1}}}}$,
\begin{multline} \label{eq:21}
|\mathfrak{a}(\xi)| \leq \frac{
\frac{s(s-1)}{2 s^{\frac{2-s}{1-s}}} \frac{|\xi|^4}{c^{\frac{2s}{1-s}}}
}
{
\frac{|\xi|^2+1}{2} \left( |\xi|^2+1 \right)
}
=
\frac{s(s-1)}{s^{\frac{2-s}{1-s}} c^{\frac{2s}{1-s}}} \frac{|\xi|^4}{\left(|\xi|^2 +1 \right)^2} \\
\leq \min \left\{
\frac{s(s-1)}{s^{\frac{2-s}{1-s}} c^{\frac{2s}{1-s}}},\frac{s(s-1)}{s^{\frac{2-s}{1-s}} c^{\frac{2s}{1-s}}} \frac{|\xi|^{2-2s}}{\left(|\xi|^2+1 \right)^{1-s}} 
\right\}
\leq \frac{s(s-1)}{s^{\frac{2-s}{1-s}}} \min\left\{
\frac{1}{c^{\frac{2s}{1-s}}},\frac{1}{c^{\frac{2s}{1-s}}} 
\frac{|\xi|^{2-2s}}{\left(|\xi|^2+1 \right)^{1-s}}
%\frac{|\xi|}{\left( |\xi|^2+1 \right)^{1/2}}
\right\} \\
\leq \frac{s(s-1)}{s^{\frac{2-s}{1-s}}} \min \left\{
\frac{1}{c^{\frac{2s}{1-s}}}, \frac{1}{c^{\frac{2s^2}{1-s}}} \left( \frac{2^{\frac{1}{1-s}}-1}{s^{\frac{1}{s-1}}} \right)^{1-s} \frac{1}{\left(|\xi|^2+1 \right)^{1-s}}
\right\}.
\end{multline}
In the previous estimate we have used the fact that
\begin{equation*}
\frac{|\xi|^4}{(|\xi|^2+1)^2} =\left(
\frac{|\xi|^{2-2s}}{(|\xi|^2+1)^{1-s}}
\right)^{\frac{2}{1-s}} \leq \frac{|\xi|^{2-2s}}{(|\xi|^2+1)^{1-s}},
\end{equation*}
which is true because $\frac{|\xi|^2}{|\xi|^2+1} \leq 1$.
On the other hand, if $|\xi| \geq c^{\frac{s}{1-s}} \sqrt{\strut
\frac{2^{\frac{1}{1-s}}-1}
{s^{\frac{1}{s-1}}}}$,
\begin{equation*}
|\mathfrak{a}(\xi)| \leq \frac{1}{P_\infty(\xi)} + \frac{1}{P_c(\xi)}
\leq \frac{1}{|\xi|^2+1} + \frac{1}{\left( \left( 1+\frac{1}{2^{\frac{1}{1-s}}-1} \right)^s - \frac{1}{\left(2^{\frac{1}{1-s}}-1 \right)^s} \right) c^{2s} |\xi|^{2s}+1}.
\end{equation*}
In particular,
\begin{equation} \label{eq:18}
|\mathfrak{a}(\xi)| \lesssim \frac{1}{c^{\frac{2s}{1-s}}+1} + \frac{1}{c^{2s} c^{\frac{2s^2}{1-s}}+1} \lesssim \frac{1}{c^{\frac{2s}{1-s}}}.
\end{equation}
We can also write
%\begin{multline} \label{eq:19}
%|\mathfrak{a}(\xi)| \leq\frac{1}{|\xi|^2+1} + \frac{1}{\left( \left( 1+\frac{1}{2^{\frac{1}{1-s}}-1} \right)^s - \frac{1}{\left(2^{\frac{1}{1-s}}-1 \right)^s} \right) c^{2s} |\xi|^{2s}+1} \\
%= \frac{1}{(|\xi|^2+1)^{1/2}} \frac{1}{(|\xi|^2+1)^{1/2}} +  \frac{(|\xi|^2+1)^{1/2}}{\left( \left( 1+\frac{1}{2^{\frac{1}{1-s}}-1} \right)^s - \frac{1}{\left(2^{\frac{1}{1-s}}-1 \right)^s} \right) c^{2s} |\xi|^{2s}+1} \frac{1}{(|\xi|^2+1)^{1/2}} \\
%\lesssim \frac{1}{c^{\frac{s}{1-s}} (|\xi|^2+1)^{1/2}}.
%\end{multline}
\begin{multline} \label{eq:19}
|\mathfrak{a}(\xi)| \leq\frac{1}{|\xi|^2+1} + \frac{1}{\left( \left( 1+\frac{1}{2^{\frac{1}{1-s}}-1} \right)^s - \frac{1}{\left(2^{\frac{1}{1-s}}-1 \right)^s} \right) c^{2s} |\xi|^{2s}+1} \\
= \frac{(|\xi|^2+1)^{1-s}}{|\xi|^2+1} \frac{1}{(|\xi|^2+1)^{1-s}} +  \frac{(|\xi|^2+1)^{1-s}}{\left( \left( 1+\frac{1}{2^{\frac{1}{1-s}}-1} \right)^s - \frac{1}{\left(2^{\frac{1}{1-s}}-1 \right)^s} \right) c^{2s} |\xi|^{2s}+1} \frac{1}{(|\xi|^2+1)^{1-s}} \\
=\frac{1}{(|\xi|^2+1)^s} \frac{1}{(|\xi|^2+1)^{1-s}} +  \frac{(|\xi|^2+1)^{1-s}}{\left( \left( 1+\frac{1}{2^{\frac{1}{1-s}}-1} \right)^s - \frac{1}{\left(2^{\frac{1}{1-s}}-1 \right)^s} \right) c^{2s} |\xi|^{2s}+1} \frac{1}{(|\xi|^2+1)^{1-s}}
\\
\lesssim \frac{1}{c^{\frac{2s^2}{1-s}}} \frac{1}{(|\xi|^2+1)^{1-s}} +
\frac{|\xi|^{2(1-s)}}{c^{2s} |\xi|^{2s}} \frac{1}{(|\xi|^2+1)^{1-s}}
\lesssim \frac{1}{c^{\frac{2s^2}{1-s}} (|\xi|^2+1)^{1-s}}.
\end{multline}
Putting together \eqref{eq:18} and \eqref{eq:19} we get also in this case
\begin{equation} \label{eq:20}
|\mathfrak{a}(\xi)| \lesssim \min \left\{
\frac{1}{c^{\frac{2s}{1-s}}},\frac{1}{c^{\frac{2s^2}{1-s}} (|\xi|^2+1)^{1-s}}
\right\}
\end{equation}
Equations \eqref{eq:21} and \eqref{eq:20} prove that \eqref{eq:22} holds for $|\alpha|=0$. 
%
%\bigskip
%
%\textcolor{red}{
%If we assume furthermore that $s \geq 1/2$, then \eqref{eq:20} can be written
%\begin{equation}
%|\mathfrak{a}(\xi)| \lesssim \min \left\{ \frac{1}{c^2},\frac{1}{c(|\xi|^2+1)^{1-s}} \right\}.
%\end{equation}
%Indeed,
%\begin{equation*}
%c^{\frac{2s}{1-s}} \geq c^2
%\end{equation*}
%and
%\begin{equation*}
%c^{\frac{2s^2}{1-s}} \geq c
%\end{equation*}
%when $1/2<s<1$.}
%
%\bigskip

If $|\alpha|=1$ we turn back to \eqref{eq:17}. 
The first term can be estimated as
\begin{multline*}
\left| \left( 2sc^2 \xi_j \left( c^2 |\xi|^2 + \frac{c^{\frac{2}{1-s}}}{s^{\frac{1}{s-1}}} \right)^{s-1}-2\xi_j \right)\frac{1}{P_c(\xi)} \frac{1}{P_\infty(\xi)} \right| \\
\lesssim \frac{2|\xi|}{(|\xi|^2+1)} \frac{1}{\left( c^2 |\xi|^2 + \frac{c^{\frac{2}{1-s}}}{s^{\frac{1}{s-1}}} \right)^s - \frac{c^{\frac{2s}{1-s}}}{s^{\frac{s}{s-1}}} +1} 
\lesssim \frac{1}{c^{\frac{2s}{1-s}} |\xi|}.
\end{multline*}
But
\begin{multline*}
\left| \left( 2sc^2 \xi_j \left( c^2 |\xi|^2 + \frac{c^{\frac{2}{1-s}}}{s^{\frac{1}{s-1}}} \right)^{s-1}-2\xi_j \right)\frac{1}{P_c(\xi)} \frac{1}{P_\infty(\xi)} \right| \\
\lesssim \frac{1}{|\xi|}
\frac{\left( c^2 |\xi|^2+\frac{c^{\frac{2}{1-s}}}{s^{\frac{1}{s-1}}} \right)^{1-s} - sc^2}{\left(c^2 |\xi|^2 + \frac{c^{\frac{2}{1-s}}}{s^{\frac{1}{s-1}}} \right)^s - \frac{c^{\frac{2s}{1-s}}}{s^{\frac{s}{s-1}}}+1}
\frac{1}{\left( c^2 |\xi|^2 + \frac{c^{\frac{2}{1-s}}}{s^{\frac{2}{2s-s}}} \right)^{1-s}} \\
\lesssim \frac{1}{|\xi|} \frac{c^2}{c^{\frac{2s}{1-s}}} \frac{1}{c^{2(1-s)}\left( |\xi|^2 +1 \right)^{1-s}} 
\lesssim \frac{1}{|\xi|} \frac{1}{c^{\frac{2s^2}{1-s}}} \frac{1}{\left( |\xi|^2 + 1 \right)^{1-s}}
\end{multline*}
by the monotonicity of the map 
\begin{equation*}
t \mapsto \frac{t^{1-s}}{t^s -\frac{c^{\frac{2s}{1-s}}}{s^{\frac{s}{s-1}}}+1} \quad \text{on the interval $\left[ \frac{c^{\frac{2}{1-s}}}{s^{\frac{1}{s-1}}},+\infty \right)$}
\end{equation*}
for $1/2<s<1$. Hence the first term can be estimated by
\begin{equation*}
\frac{1}{|\xi|} \min \left\{ \frac{1}{c^{\frac{2s}{1-s}}},\frac{1}{c^{\frac{2s^2}{1-s}}} \frac{1}{\left( |\xi|^2 + 1 \right)^{1-s}} \right\}.
\end{equation*}
By using \eqref{eq:23} and \eqref{eq:20} it is easy to check that the other terms in \eqref{eq:17} satisfy the same estimate.

We conclude by induction. Indeed, when we compute a term like 
\begin{equation*}
\partial_{\xi_j} D^\alpha \mathfrak{a}, 
\end{equation*}
we just differentiate sum of products of terms as described above. If we differentiate a polynomial, the total degree is reduced by one. Otherwise, see \eqref{eq:24} and \eqref{eq:23}, some extra terms appear that are estimated by $1/|\xi|$. This completes the induction step, and the proof of \eqref{eq:22}.

The proof of \eqref{eq:25} is rather similar, and we only sketch the main steps. First of all, \eqref{eq:26}, which is valid for \emph{every} $\xi \in \mathbb{R}^N$ since $Df(t) \leq s$ for every $t \geq 0$, immediately implies that
\begin{equation} \label{eq:27}
\left| \frac{P_c(\xi)}{P_\infty(\xi)} \right| \leq 1 \quad \text{for every $\xi \in \mathbb{R}^N$}.
\end{equation}
Then
\begin{equation*}
\partial_{\xi_j} \left( \frac{P_c(\xi)}{P_\infty(\xi)} \right) = \frac{1}{P_\infty(\xi)} \frac{2sc^2 \xi_j}{\left(c^2 |\xi|^2 + \frac{c^{\frac{2}{1-s}}}{s^{\frac{1}{s-1}}} \right)^{1-s}} - \frac{P_c(\xi)}{P_\infty(\xi)} \frac{2\xi_j}{P_\infty(\xi)}
\end{equation*}
is dominated by some multiple of
\begin{equation*}
\frac{|\xi|}{P_\infty(\xi)} + \frac{|\xi|}{P_\infty(\xi)} \cdot 1 \leq \frac{1}{|\xi|}.
\end{equation*}
thanks to \eqref{eq:27}.

Then we check that $D^\alpha (P_c/P_\infty)$ can be expressed as a sum
of products of terms like a polynomial of $c$, $\xi_1$,\ldots,
$\xi_N$, $P_c/P_\infty$, $1/P_\infty^{\ell_1}$, and
$\left(c^2 |\xi|^2 + \frac{c^{\frac{2}{s-2}}}{s^{\frac{1}{s-1}}}
\right)^{s-\ell_2}$, where $\ell_1$, $\ell_2 \in \mathbb{N}$. Using
again \eqref{eq:23} and the induction hypothesis, we can prove that
\begin{equation*}
\left| \partial_{\xi_j} D^\alpha \left( \frac{P_c(\xi)}{P_\infty(\xi)} \right) \right| \lesssim \frac{1}{|\xi|^{|\alpha|+1}}
\end{equation*}
\end{proof}
We recall a celebrated result, see \cite{Mikhlin57} for the original
proof.
\begin{theorem}[H\"{o}rmander-Mikhlin] \label{th:1.3} Suppose that
  $\mathfrak{m}\colon \mathbb{R}^N \to \mathbb{R}$ satisfies
	\begin{equation*}
          \left| D^\alpha \mathfrak{m}(\xi) \right| \leq \frac{B_\alpha}{|\xi|^{|\alpha|}}
      \end{equation*}
      for all multi-indices
      $\alpha \in \left( \mathbb{N}\cup \{0\} \right)^N$ such that
      $0 \leq |\alpha| \leq N/2+1$. Then for any $1<q<\infty$, there
      exists a constant $C=C(q,N)>0$ such that
	\begin{equation*}
          \left\| \mathfrak{m}(D)f \right\|_{L^q} \leq C \left(\sup_{0 \leq |\alpha| \leq \frac{N}{2}+1} B_\alpha \right) \|f\|_{L^q}.
	\end{equation*}
\end{theorem}
This fundamental result allows us to transform the analytic inequalities of Proposition \ref{prop:1.4} into useful operator inequalities.
\begin{theorem} \label{th:1.5}
  (1) For $1<q<\infty$, there exists a
  constant $C=C(q,N)>0$ such that
  \begin{equation} \label{eq:28}
    \left\| \left( \frac{1}{P_\infty(D)}
        -\frac{1}{P_c(D)} \right) f \right\|_{L^q} \leq
    \frac{C}{c^{\frac{2s}{1-s}}} \left\| f \right\|_{L^q}.
\end{equation}
(2) For $1<q<\infty$, there exists a constant $C=C(q,N)>0$ such that
\begin{equation} \label{eq:29}
  \left\| \left( \frac{1}{P_\infty(D)}
      -\frac{1}{P_c(D)} \right) f \right\|_{L^q} \leq
  \frac{C}{c^{\frac{2s^2}{1-s}}} \left\| \frac{1}{P_\infty(D)^{1-s}} f
  \right\|_{L^q}.
\end{equation}
\end{theorem}
\begin{proof}
  Pick any multi-index
  $\alpha \in \left( \mathbb{N} \cup \{0\} \right)^N$. From
  Proposition \ref{prop:1.4} we derive that
\begin{equation*}
\left| D^\alpha \left( \frac{1}{P_{\infty}(\xi)} - \frac{1}{P_c(\xi)} \right) \right| \lesssim \frac{1}{c^{\frac{2s}{1-s}} |\xi|^{|\alpha|}}
\end{equation*}
and
\begin{multline*}
\left| D^{\alpha} \left( \left( \frac{1}{P_\infty(\xi)} - \frac{1}{P_c(\xi)} \right) P_\infty(\xi)^{1-s} \right) \right| \\
\leq \sum_{\alpha_1+\alpha_2=\alpha} \left| D^{\alpha_1} \left( \frac{1}{P_\infty(\xi)} - \frac{1}{P_c(\xi)} \right) \right| \left| D^{\alpha_2} \left( P_\infty (\xi)^{1-s} \right) \right|
\\
\lesssim \sum_{\alpha_1+\alpha_2=\alpha} \frac{1}{c^{\frac{2s^2}{1-s}} |\xi|^{|\alpha_1|+2-2s}} \frac{1}{|\xi|^{|\alpha_2|+2s-2}} 
\lesssim \frac{1}{c^{\frac{2s^2}{1-s}}} \frac{1}{|\xi|^{|\alpha|}}.
\end{multline*}
We conclude by Theorem \ref{th:1.3}.
\end{proof}
We conclude this Section with a statement that we will use to set up
our fixed point argument. 
\begin{theorem} \label{th:1.6}
For $1<q<\infty$ there exists a constant $C=C(q,N)>0$ such that for $c \gg 1$,
\begin{equation*}
C^{-1} \|f\|_{W^{1,q}} \leq \left\| P_c(D) f \right\|_{L^q} \leq C \|f\|_{W^{2,q}}.
\end{equation*}
\end{theorem}
\begin{proof}
By the triangle inequality and  Theorem \ref{th:1.5} we can write
\begin{align*}
\left\| \frac{1}{P_c(D)}f \right\|_{L^q} &\leq  \left\| \frac{1}{P_\infty(\xi)} f \right\|_{L^q} + \left\| \left( \frac{1}{P_\infty(D)}-\frac{1}{P_c(D)} \right) f \right\|_{L^q}
\\
&\lesssim \left\| \frac{1}{P_\infty(\xi)} f \right\|_{L^q} + \frac{1}{c^{\frac{2s^2}{1-s}}} \left\| \frac{1}{P_\infty(D)^{1-s}} f \right\|_{L^q}
\\
&\lesssim \left\| f \right\|_{W^{-2,q}} + \left\| f \right\|_{W^{2s-2,q}}
\\
&\lesssim \left\| f \right\|_{W^{-2,q}} \lesssim \left\| f \right\|_{W^{-1,q}}
\end{align*}
by the Sobolev embedding $W^{2-2s,q} \subset W^{2,q}$. 

If we replace $f$ by $P_c(D)\sqrt{P_\infty(D)} f = \sqrt{P_\infty(D)} P_c(D)f$ we get
\begin{equation*}
\left\| f \right\|_{W^{1,q}} \lesssim
\left\| \sqrt{P_\infty(D)} f \right\|_{L^q} \lesssim \left\| \sqrt{P_\infty(D)} P_c(D) f \right\|_{W^{-1,q}} \lesssim \left\| P_c(D) f \right\|_{L^q}.
\end{equation*}
The other inequality follows from \eqref{eq:25} and Theorem \ref{th:1.3}:
\begin{equation*}
\left\| P_c(D)f \right\|_{L^q} = \left\| \frac{P_c(D)}{P_\infty(D)} P_\infty(D)f \right\|_{L^q} \lesssim \left\| P_\infty(D)f \right\|_{L^q} \lesssim \left\| f \right\|_{W^{2,q}}.
\end{equation*}
\end{proof}

\section{A fixed-point argument}

We revert to the problem of finding a fixed point for the operator $\Phi_c$ defined in \eqref{eq:30}.
\begin{definition}
	If $X$ and $Y$ are two Banach spaces, the norm of $X \cap Y$ is defined to be $\|\cdot\|_{X \cap Y} = \max\{ \|\cdot \|_X,\| \cdot \|_Y\}$.
\end{definition}

\begin{remark} \label{rem:2.3}
	We will use the fact that the non-relativistic ground state $u_\infty$ is positive, radially symmetric (about the origin without loss of generality), and non-degenerate in the subspace of radially symmetric functions (see \cite{Comech14,Ounaies07}), namely
\begin{equation*}
\ker \mathcal{L}_\infty \cap H^1_{\mathrm{rad}} = \{ 0\},
\end{equation*}
where
\begin{equation*}
\mathcal{L}_\infty = -\Delta +1 - p u_\infty^{p-1} \colon H^2 \to L^2.
\end{equation*}
\end{remark}
\begin{lemma} \label{lem:2.2}
For any $2 \leq q <\infty$, the operator 
\begin{equation} \label{eq:37}
\mathcal{A} = I - p u_\infty^{p-1} P_\infty(D)^{-1}
\end{equation}
is invertible from $L^2_{\mathrm{rad}} \cap L^q$ into $L_{\mathrm{rad}}^2 \cap L^q$.
\end{lemma}
\begin{proof}
Since the ground state $u_\infty$ decays exponentially fast at infinity (see \cite{Berestycki83}), it is easy to check that the operator $p u_\infty^{p-1} P_\infty(D)^{-1}$ is compact as the composition of a compact multiplication operator and a bounded operator. By the Fredholm alternative, it suffices to show that $\mathcal{A}$ is injective. But if $v \in \ker \mathcal{A}$, then $P_\infty(D)^{-1}v \in \ker \mathcal{L}_\infty$. It follows from Remark \ref{rem:2.3} that $P_\infty(D)^{-1}v =0$, or $v=0$.
\end{proof}
Now we can write
\begin{align} \label{eq:32}
\mathcal{L}_{c,\infty} &= \left( I - p u_\infty^{p-1} P_c(D)^{-1} \right) P_c(D) \nonumber \\
&= \left( \mathcal{A}+p u_\infty^{p-1} \left( P_\infty(D)^{-1}-P_c(D)^{-1} \right) \right)P_c(D) \nonumber \\
&= \left( I + p u_\infty^{-1} \left( P_\infty(D)^{-1} - P_c(D)^{-1} \right) \mathcal{A}^{-1} \right) \mathcal{A} P_c(D),
\end{align}
where $\mathcal{A}$ was defined in \eqref{eq:37}.
\begin{lemma} \label{lem:2.3}
Suppose that $1<p<\infty$ if $N=1$, $2$, and $1<p<(N+2)/(N-2)$ if $N \geq 3$. Then there exists $c_0>0$ such that for every $c \geq c_0$ there results
\begin{equation*}
\left\| p u_\infty^{p-1} \left( P_\infty(D)^{-1} - P_c(D)^{-1} \right) \mathcal{A}^{-1} \right\|_{L(L_{\mathrm{rad}}^2 \cap L^q)} \leq \frac{1}{2}.
\end{equation*}
\end{lemma}
\begin{proof}
We start with the simple estimate
\begin{multline} \label{eq:31}
\left\| p u_\infty^{p-1} \left( P_\infty(D)^{-1} - P_c(D)^{-1} \right) \mathcal{A}^{-1} \right\|_{L(L_{\mathrm{rad}}^2 \cap L^q)} \leq \\
p \|u_\infty^{p-1} I \|_{L(L_{\mathrm{rad}}^2 \cap L^q)} \|P_\infty(D)^{-1} - P_c(D)^{-1} \|_{L(L_{\mathrm{rad}}^2 \cap L^q)} \|\mathcal{A}^{-1}\|_{L(L_{\mathrm{rad}}^2 \cap L^q)}.
\end{multline}
H\"{o}lder's inequality implies that $\|u_\infty^{p-1} I \|_{L(L_{\mathrm{rad}}^2 \cap L^q)} \leq \|u_\infty\|_{L^\infty}^{p-1}$. By \eqref{eq:28}, 
\begin{equation*}
\|P_\infty(D)^{-1} - P_c(D)^{-1} \|_{L(L_{\mathrm{rad}}^2 \cap L^q)}  \lesssim c^{-2}.
\end{equation*}
 By Lemma \ref{lem:2.2} $\|\mathcal{A}^{-1}\|_{L(L_{\mathrm{rad}}^2 \cap L^q)} \leq \infty$. We conclude the proof by inserting these estimates into \eqref{eq:31}.
\end{proof}
We can now proceed with the proof of the invertibility of $\mathcal{L}_{c,\infty}$.
\begin{proposition} \label{prop:2.1}
Let $2 \leq q < \infty$. For every $c>0$ sufficiently large, the operator
\begin{equation*}
\mathcal{L}_{c,\infty} \colon H^1_{\mathrm{rad}} \cap W^{1,q} \to L^2_{\mathrm{rad}} \cap L^q
\end{equation*}
is invertible. Furthermore, its inverse is uniformly bounded in the sense that
\begin{equation*}
\sup_{c >0} \left\| \left( \mathcal{L}_{c,\infty} \right)^{-1} \right\|_{L(H^1_{\mathrm{rad}} \cap W^{1,q},L^2_{\mathrm{rad}} \cap L^q
)} < \infty,
\end{equation*}
where $L(H^1_{\mathrm{rad}} \cap W^{1,q},L^2_{\mathrm{rad}} \cap L^q
)$ is the Banach space of continuous linear operators with its standard norm.
\end{proposition}
\begin{proof}
By \eqref{eq:32} and the previous Lemmas, the operator $\mathcal{L}_{c,\infty}$ is invertible for $c \geq c_0$, and
\begin{equation*}
\mathcal{L}_{c,\infty}^{-1} = P_c(D)^{-1} \mathcal{A}^{-1} \left( I +p u_\infty^{p-1} \left(P_\infty(D)^{-1}-P_c(D)^{-1} \right) \mathcal{A}^{-1} \right)^{-1}.
\end{equation*}
Moreover, by Theorem \ref{th:1.6} and Lemmas \ref{lem:2.2} and \ref{lem:2.3} we have
\begin{multline*}
\left\| \mathcal{L}_{c,\infty}^{-1} \right\|_{L(L_{\mathrm{rad}}^2 \cap L^q,H_{\mathrm{rad}}^1 \cap W^{1,q})} \\
\leq \left\| P_c(D)^{-1} \right\|_{L(L_{\mathrm{rad}}^2 \cap L^q , H_{\mathrm{rad}}^1 \cap W^{1,q})} \cdot \left\| \left( I +p u_\infty^{p-1} \left(P_\infty(D)^{-1}-P_c(D)^{-1} \right) \mathcal{A}^{-1} \right)^{-1} \right\|_{L(L^2_{\mathrm{rad}}\cap L^q)} 
\leq C,
\end{multline*}
where $C>0$ is a constant \emph{independent} of $c \geq c_0$.
\end{proof}

\bigskip

To prove that $\Phi_c$ is a contraction (in some suitable space), we must provide  bounds for its terms. Recalling that
\begin{equation*}
\Phi_c(w) = \left(\mathcal{L}_{c,\infty} \right)^{-1} \left( P_\infty(D)-P_c(D) \right)u_\infty + \left(\mathcal{L}_{c,\infty} \right)^{-1} Q(w),
\end{equation*}
we first estimate for large $c$ the quantity
\begin{equation*}
\mathcal{R}_c = \left(\mathcal{L}_{c,\infty} \right)^{-1} \left( P_\infty(D)-P_c(D) \right)u_\infty.
\end{equation*}
\begin{lemma} \label{lem:2.4}
Let $2 \leq q <\infty$. Then we have
\begin{equation*}
\|\mathcal{R}_c \|_{H^1_{\mathrm{rad}} \cap W^{1,q}} =
\begin{cases}
O(c^{-\frac{2s^2}{1-s}}) &\text{if $1<p \leq 2$}\\
O(c^{-\frac{2s}{1-s}}) &\text{if $p>2$}.
\end{cases}
\end{equation*}
\end{lemma}
\begin{proof}
It is well known that $u_\infty \in H_{\mathrm{rad}}^{2+\lfloor p \rfloor} \cap W_{\mathrm{rad}}^{2+\lfloor p \rfloor,q}$, where $\lfloor p \rfloor$ is the largest integer less than or equal to $p$. If $p>2$, by Proposition \ref{prop:2.1}, \eqref{eq:28} and Theorem \ref{th:1.6} we deduce that
\begin{multline*}
\left\|\mathcal{R}_c \right\|_{H^1_{\mathrm{rad}} \cap W^{1,q}} \\
\leq \left\| \mathcal{L}_{c,\infty}^{-1} \right\|_{L(L_{\mathrm{rad}}^2 \cap L^q,H_{\mathrm{rad}}^1 \cap W^{1,q})} \left\| \frac{P_\infty(D)-P_c(D)}{P_\infty(D) P_c(D)} \right\|_{L(L_{\mathrm{rad}}^2 \cap L^q)} \left\| P_\infty(D) P_c(D) u_\infty \right\|_{L_{\mathrm{rad}}^2 \cap L^q} \\
\lesssim c^{-\frac{2s}{1-s}} \left\| P_\infty(D) u_\infty \right\|_{H_{\mathrm{rad}}^2 \cap W_{\mathrm{rad}}^{2,q}} \lesssim c^{-\frac{2s}{1-s}} \|u_\infty\|_{H_{\mathrm{rad}}^4 \cap W_{\mathrm{rad}}^{4,q}}.
\end{multline*}
Analogously, if $1<p<2$, again by Proposition \ref{prop:2.1}, \eqref{eq:29} and Theorem \ref{th:1.6} we deduce that
\begin{multline*}
\left\|\mathcal{R}_c \right\|_{H^1_{\mathrm{rad}} \cap W^{1,q}} \\
\leq \left\| \mathcal{L}_{c,\infty}^{-1} \right\|_{L(L_{\mathrm{rad}}^2 \cap L^q,H_{\mathrm{rad}}^1 \cap W^{1,q})} \left\| \frac{P_\infty(D)-P_c(D)}{P_\infty(D)^{1-s} P_c(D)} \right\|_{L(L_{\mathrm{rad}}^2 \cap L^q)} \left\| P_\infty(D)^{1-s} P_c(D) u_\infty \right\|_{L_{\mathrm{rad}}^2 \cap L^q} \\
\lesssim c^{-\frac{2s^2}{1-s}} \left\| P_\infty(D)^{1-s} u_\infty \right\|_{H_{\mathrm{rad}}^2 \cap W_{\mathrm{rad}}^{2,q}} \lesssim c^{-\frac{2s^2}{1-s}} \|u_\infty\|_{H_{\mathrm{rad}}^3 \cap W_{\mathrm{rad}}^{3,q}},
\end{multline*}
since $1/2<s<1$.
\end{proof}
We turn to the estimate of the second term in the formula of $\Phi_c$.
\begin{lemma} \label{lem:2.5}
Fix $q>N$ and suppose that $0<\delta \leq \|u_\infty\|_{H^1}$. Then for $c \geq c_0$ we have
\begin{align}
\left\| \mathcal{L}_{c,\infty}^{-1} Q(w) \right\|_{H_{\mathrm{rad}}^1 \cap W^{1,q}} &\lesssim \delta^{\min \{p,2 \}} \label{eq:33} \\
\left\| \mathcal{L}_{c,\infty}^{-1} Q(w) - \mathcal{L}_{c,\infty}^{-1} Q(\tilde{w}) \right\|_{H_{\mathrm{rad}}^1 \cap W^{1,q}}  &\lesssim \delta^{\min \{p-1,1\}} \|w-\tilde{w}\|_{H_{\mathrm{rad}}^1 \cap W^{1,q}}. \label{eq:34}
\end{align}
\end{lemma}
\begin{proof}
Clearly \eqref{eq:33} follows from \eqref{eq:34} by choosing $\tilde{w}=0$. Hence we focus on the second estimate. By definition of $Q$, we write
\begin{multline*}
Q(w) - Q(\tilde{w}) = \left( |u_\infty + w|^{p-1}(u_\infty+w) - |u_\infty +\tilde{w}|^{p-1} (u_\infty+\tilde{w}) \right) - p u_\infty^{p-1} (w-\tilde{w}) \\
= \int_0^1 \frac{d}{dt} \left( |u_\infty + (1-t) \tilde{w}+tw|^{p-1} \left( u_\infty + (1-t) \tilde{w} + tw \right) \right)\, dt - p u_\infty^{p-1} (w-\tilde{w}) \\
= p \int_0^1 \left( |u_\infty + (1-t) \tilde{w} + tw |^{p-1} \right) (w-\tilde{w}) \, dt.
\end{multline*}
If $1<p \leq 2$, we conclude that $|Q(w)-Q(\tilde{w})| \leq C (|w|+|\tilde{w}|)^{p-1} |w-\tilde{w}|$ by the elementary inequality $||a|^\ell - |b|^\ell| \leq ||a|-|b||^\ell \leq |a-b|^\ell$ for $0<\ell <1$. By Proposition \ref{prop:2.1} and the Sobolev embedding $W^{1,q} \subset L^\infty$,
\begin{multline*}
\left\| \mathcal{L}_{c,\infty}^{-1} Q(w) - \mathcal{L}_{c,\infty}^{-1} Q(\tilde{w}) \right\|_{H_{\mathrm{rad}}^1 \cap W^{1,q}}  \leq C \left\| Q(w)-Q(\tilde{w}) \right\|_{L_{\mathrm{rad}}^2 \cap L^q} \\
\leq C \left( \|w\|_{H_{\mathrm{rad}}^1 \cap W^{1,q}} + \| \tilde{w} \|_{H_{\mathrm{rad}}^1 \cap W^{1,q}} \right)^{p-1} \|w-\tilde{w}\|_{H_{\mathrm{rad}}^1 \cap W^{1,q}}.
\end{multline*}
If $p>2$, we proceed as before and  conclude that
\begin{equation*}
|Q(w)-Q(\tilde{w})| \leq C \left( u_\infty + |w| + |\tilde{w}| \right)^{p-2} (|w|+|\tilde{w}|) |w-\tilde{w}|.
\end{equation*}
This yields as above
\begin{multline*}
\left\| \mathcal{L}_{c,\infty}^{-1} Q(w) - \mathcal{L}_{c,\infty}^{-1} Q(\tilde{w}) \right\|_{H_{\mathrm{rad}}^1 \cap W^{1,q}}  \\
\leq C \left\| Q(w)-Q(\tilde{w}) \right\|_{L_{\mathrm{rad}}^2 \cap L^q} 
\leq C \left\| (u_\infty + |w| +|\tilde{w}|)^{p-2} (|w|+|\tilde{w}|)|w-\tilde{w}| \right\|_{L_{\mathrm{rad}}^2 \cap L^q} \\
\leq C \left(
\|u_\infty\|_{H_{\mathrm{rad}}^1 \cap W^{1,q}} + \|w\|_{H_{\mathrm{rad}}^1 \cap W^{1,q}} + \| \tilde{w} \|_{H_{\mathrm{rad}}^1 \cap W^{1,q}}
\right)^{p-2} \cdot \\
\cdot \left(
\|w\|_{H_{\mathrm{rad}}^1 \cap W^{1,q}} + \| \tilde{w} \|_{H_{\mathrm{rad}}^1 \cap W^{1,q}}
\right) \|w-\tilde{w}\|_{H_{\mathrm{rad}}^1 \cap W^{1,q}}.
\end{multline*}
The proof is complete.
\end{proof}
\begin{proposition} \label{prop:3.8}
Let $q>N$. For any $\delta>0$ sufficiently small, there exists $c_0>0$ such that, if $c \geq c_0$, then $\Phi_c$ has a unique fixed point in the (closed) ball 
\begin{equation*}
B_\delta = \left\{ w \in H_{\mathrm{rad}}^1 \cap W^{1,q} \mid \|w\|_{H_{\mathrm{rad}}^1 \cap W^{1,q}} \leq \delta \right\}.
\end{equation*}
\end{proposition}
\begin{proof}
By Lemma \ref{lem:2.4} we choose $c_0$ so large that $\|\mathcal{R}_c\|_{H_{\mathrm{rad}}^1 \cap W^{1,q}} \leq \delta/2$ for $c \geq c_0$. By Lemma \ref{lem:2.5}, $w$, $\tilde{w} \in B_\delta$ implies $\|\Phi_c(w)\|_{H_{\mathrm{rad}}^1\cap W^{1,q}} \leq \delta/2 + C \delta^{\min \{p,2\}} \leq \delta$ and 
\begin{equation*}
\left\| \Phi_c(w)-\Phi_c(\tilde{s}) \right\|_{H_{\mathrm{rad}}^1 \cap W^{1,q}} \leq C \delta^{\min \{p-1,1\}} \|w-\tilde{w}\|_{H_{\mathrm{rad}}^1 \cap W^{1,q}} \leq \frac{1}{2} \|w-\tilde{w}\|_{H_{\mathrm{rad}}^1 \cap W^{1,q}}
\end{equation*}
as soon as $C^{\min \{p-1,1\}} \delta \leq 1/2$. A application of the Contraction Theorem (see \cite{Banach22}) yields the result.
\end{proof}
\begin{lemma} \label{lem:3.9}
For $w$ the fixed point $w$ constructed in Proposition \ref{prop:3.8}, the function $u=u_\infty+w$  is a solution of the equation $P_c(D)u=|u|^{p-1}u$.
\end{lemma}
\begin{proof}
Indeed, we alrady know that $w = \mathcal{R}_c + (\mathcal{L}_{c,\infty})^{-1} Q(w)$ in $H_{\mathrm{rad}}^1 \cap W^{1,q}$, and thus
\begin{multline*}
0 = \mathcal{L}_{c,\infty} w - \mathcal{L}_{c,\infty} \mathcal{R}_c - Q(w) \\
= \left( P_c(D)-p u_\infty^{p-1} \right)w - \left( P_\infty(D) - P_c(D) \right) u_\infty - \left( |u|^{p-1}u - u_\infty^p - p u_\infty^{p-1}w \right) \\
= P_c(D)w - pu_\infty^{p-1} w - P_\infty(D) u_\infty + P_c(D) u_\infty - |u|^{p-1}u + u_\infty^p + p u_\infty^{p-1} w 
= P_c(D) u - |u|^{p-1}u.
\end{multline*}
\end{proof}
The proof of Theorem \ref{th:main} follows immediately from the previous Lemma and a rescaling, see (\ref{eq:36}).
\begin{corollary}
Let $u_c$ be a solution of (\ref{eq:1}) in $H_{\mathrm{rad}}^1 \cap L^\infty$ that converges to $u_\infty$ as $c \to +\infty$. Then for sufficiently large $c \geq 1$, $u_c$ is unique and moreover
\begin{equation} \label{eq:35}
\left\| u_c - u_\infty \right\|_{H^1 \cap W^{1,q}} =
\begin{cases}
O(c^{-\frac{2s^2}{1-s}}) &\text{if $1<p \leq 2$} \\
O(c^{-\frac{2s}{1-s}}) &\text{if $p >2$}.
\end{cases}
\end{equation}
\end{corollary}
\begin{proof}
First of all, we claim that there exists $\delta>0$ so small that the solution to \eqref{eq:1} is unique in $B_\delta(u_\infty) = \left\{ u \in H_{\mathrm{rad}}^1 \cap L^\infty \mid \|u-u_\infty \|_{H^1 \cap L^\infty} <\delta \right\}$. 

Indeed, by Lemma \ref{lem:3.9}, $w=u_c-u_\infty$ is a fixed point of $\Phi_c$. Since we are assuming that $q>N$, the norms $\|\cdot \|_{H^1 \cap L^\infty}$ and $\| \cdot \|_{H^1 \cap W^{1,q}}$ are equivalent. The claim follows from the uniqueness of the fixed point of $\Phi_c$ in a small ball around zero.

Furthermore, by Lemma \ref{lem:2.4} we can take $\delta \sim c^{-\alpha}$ such that $\|\mathcal{R}_c \|_{H_{\mathrm{rad}}^1 \cap W^{1,q}} \leq \delta$, where $\alpha$ is either equal to $2s^2/(1-s)$ or $2s/(1-s)$. As before, we can prove that $\Phi_c$ is a contraction in a (closed) ball of radius $\sim c^{-\alpha}$, so that it admits a fixed point there. By uniqueness, taking $c$ larger if needed, this fixed point must be equal to the fixed point $w$ already constructed. Therefore the difference $u_c - u_\infty$ must satisfy \eqref{eq:35}.
\end{proof}

\end{document}